\documentclass[12pt]{article}
\usepackage{amsmath,amssymb,amsbsy,amsfonts,amsthm}

\begin{document}

\newtheorem{theorem}{Theorem}
\newtheorem{lemma}[theorem]{Lemma}
\newtheorem{cor}[theorem]{Corollary}
\newtheorem{prop}[theorem]{Proposition}

\newcommand{\comm}[1]{\marginpar{%
\vskip-\baselineskip 
\raggedright\footnotesize
\itshape\hrule\smallskip#1\par\smallskip\hrule}}

\def\cA{{\mathcal A}}
\def\cB{{\mathcal B}}
\def\cC{{\mathcal C}}
\def\cD{{\mathcal D}}
\def\cE{{\mathcal E}}
\def\cF{{\mathcal F}}
\def\cG{{\mathcal G}}
\def\cH{{\mathcal H}}
\def\cI{{\mathcal I}}
\def\cJ{{\mathcal J}}
\def\cK{{\mathcal K}}
\def\cL{{\mathcal L}}
\def\cM{{\mathcal M}}
\def\cN{{\mathcal N}}
\def\cO{{\mathcal O}}
\def\cP{{\mathcal P}}
\def\cQ{{\mathcal Q}}
\def\cR{{\mathcal R}}
\def\cS{{\mathcal S}}
\def\cT{{\mathcal T}}
\def\cU{{\mathcal U}}
\def\cV{{\mathcal V}}
\def\cW{{\mathcal W}}
\def\cX{{\mathcal X}}
\def\cY{{\mathcal Y}}
\def\cZ{{\mathcal Z}}

\def\C{\mathbb{C}}
\def\F{\mathbb{F}}
\def\K{\mathbb{K}}
\def\Z{\mathbb{Z}}
\def\R{\mathbb{R}}
\def\Q{\mathbb{Q}}
\def\N{\mathbb{N}}

\def\({\left(}
\def\){\right)}
\def\[{\left[}
\def\]{\right]}
\def\<{\langle}
\def\>{\rangle}

\def\e{e}

\def\eq{\e_q}
\def\eT{\e_T}

\def\fl#1{\left\lfloor#1\right\rfloor}
\def\rf#1{\left\lceil#1\right\rceil}
\def\mand{\qquad\mbox{and}\qquad}

\title{\bf On a bound of Heath-Brown for Dirichlet $L$-functions on the critical line}

\date{ }
\author{
{\sc   Bryce Kerr} \\
{Department of Computing, Macquarie University} \\
{Sydney, NSW 2109, Australia} \\
{\tt  bryce.kerr@mq.edu.au}}

\date{}

\maketitle

\begin{abstract}
Let  $\chi$ a primitive character$\pmod q$ and consider the Dirichlet $L$-function $$L(s,\chi)=\sum_{n=1}^{\infty}\frac{\chi(n)}{n^s}.$$
We give a new proof of an upper bound of Heath-Brown for $|L(s,\chi)|$ on the critical line $s=1/2+it$.\end{abstract}


\section{Introduction}
For integer $q$, let $\chi$ be a primitive character$\pmod q$. We consider the order of the Diriclet $L$-function
$$L(s,\chi)=\sum_{n=1}^{\infty}\frac{\chi(n)}{n^s},$$
on the critical line $s=1/2+it$. This problem has been considered in a number of previous works and is essentially equivalent to bounding sums of the form
\begin{equation}
\label{main sums}
\sum_{M<n\le M+N}\chi(n)n^{it}.
\end{equation}
Burgess~\cite{Burg3} showed that for bounded $t$
$$L(1/2+it,\chi)\ll q^{3/16+o(1)},$$
which was improved  by Heath-Brown in~\cite{HB1} by developing a $q$-analouge of the van der Corput method, to obtain the bound 
\begin{equation}
\label{HHB1}
L(1/2+it,\chi)\ll \left(q^{1/4}+(qt)^{1/6}\right)(qt)^{o(1)}.
\end{equation}
Later in~\cite{HB2}, Heath-Brown  adapted the Burgess method~\cite{Burg3} to deal with sums of the form~\eqref{main sums}, leading to the bound
\begin{equation}
\label{HHB2}
L(1/2+it,\chi) \ll (qt)^{3/16+o(1)},
\end{equation}
which gives an improvement on~\eqref{HHB1} for values of $t\le q^{1/3}$.
When $q$ is prime, Huxley and Watt~\cite{HuWt} extended the Bombieri and Iwaniec method~\cite{BoIw} to deal with sums of the form~\eqref{main sums}. This gives an improvement on~\eqref{HHB1} and~\eqref{HHB2} when $t$ is not too small, although  they remark it does not seem possible to use the method of~\cite{HuWt} to improve on the results of Heath-Brown when $t\le q^{15}$. We give a new proof of the bound~\eqref{HHB2}. As in~\cite{HB2}, we reduce the problem to bounding the double mean value 
\begin{align}
\label{double mean value1}
\int_{A}^{B}\sum_{\lambda=1}^{q}\max_{V/2<Q\le V}\left|\sum_{V/2<v\le Q}\chi(\lambda+v)(x+v)^{it}e^{2\pi i \alpha v} \right|^4dx.
\end{align}
This is done using  some ideas  of Chang~\cite{Chang1}, Friedlander and Iwaniec~\cite{FI} and Heath-Brown~\cite{HB2}. We note that although this part of the argument is different to that of Heath-Brown, we still rely on some key ideas from~\cite{HB2}, mainly the use of the Sobolev-Gallagher inequality (see Lemma~\ref{sg}).
To bound the double mean value~\eqref{double mean value1}, we follow the argument of  Heath-Brown~\cite{HB2} to reduce the problem to bounding
\begin{align*}
\int_{A}^{B}\sum_{\lambda=1}^{q}\left|\sum_{V/2<v\le V/2+C}\chi(\lambda+v)(x+v)^{it}e^{2\pi i \alpha v} \right|^4dx,
\end{align*}
 which we dealt with differently to Heath-Brown~\cite{HB2}, using results of Burgess~\cite{Burg1}, Huxley~\cite{Hu} and Phong, Stein and Sturm~\cite{PSS}. 
\section{Main Results}
\begin{theorem}
\label{main1}
For integer $q$, let $\chi$ be a primitive character$\pmod q$. Let $t\ge 1$ and let the integers $M$ and $N$ satisfy 
$$q^{1/4}t^{1/4}\le N\le q^{5/8}t^{1/8},$$
and $$N\le M \le 2N.$$  Then we have
$$\left|\sum_{M<n\le M+N}\chi(n)n^{it}\right| \le N^{1/2}(qt)^{3/16+o(1)}.$$
\end{theorem}
Using Theorem~\ref{main1} as in~\cite{HB2} gives,
\begin{theorem}
For integer $q$, let $\chi$ be a primitive character $\pmod q$ and let $t\ge 1$. Then we have
$$L(1/2+it,\chi)\le (qt)^{3/16+o(1)}.$$
\end{theorem}

\section{Prelimanary Results}
The proof of the following can be found in the proof of~\cite[Lemma~7]{FGS}.
\begin{lemma}
\label{congruence}
For integers $q,M,N,U$ let $I$ denote the number of solutions to the congruence
$$n_1u_1 \equiv n_2u_2 \pmod q,$$
with  
$$1\le u_1,u_2 \le U, \quad M<n_1,n_2 \le M+N,$$
and 
$$(u_1,q)=1, \quad  (u_2,q)=1.$$
Then 
$$I\le NU\left(\frac{NU}{q}+1\right)q^{o(1)}.$$
\end{lemma}
As in~\cite{HB2} we use the Sobolev-Gallagher inequality, for the proof see~\cite[Lemma 1.1]{Mg}.
\begin{lemma}
\label{sg}
Let $b>a$ and suppose $f(x)$ has continuous derivative for $a\le x \le b$. For any $a\le u \le b$ we have 
$$|f(u)|\le \frac{1}{b-a}\int_{a}^{b}|f(x)|dx+\int_{a}^{b}|f'(x)|dx.$$
\end{lemma}
The following is a special case of~\cite[Theorem~1]{PSS}. We reproduce the proof for the special case relevant to us.
\begin{lemma}
\label{poly small}
Let $F(x)=Lx^2+Mx+N$ be a polynomial of degree $2$ with real coefficients and two distinct roots $\alpha_1,\alpha_2$. Then
$$\mu \left( \{ x \in \mathbb{R} : |F(x)|\le \varepsilon \} \right) \le \frac{8\varepsilon}{|L(\alpha_1-\alpha_2)|},$$
where $\mu(.)$ denotes the Lebesgue measure.
 \end{lemma}
\begin{proof}
Writing
$$F(x)=L(x-\alpha_1)(x-\alpha_2),$$
we split $\mathbb{R}$ into two sets
\begin{align*}
I_1&=\{ x\in \mathbb{R} : |x-\alpha_1|\le |x-\alpha_2| \}, \\
I_2&=\{ x\in \mathbb{R} : |x-\alpha_2|\le |x-\alpha_1| \}.
\end{align*}
Then if $x\in I_1$ we have
$$|\alpha_1-\alpha_2|\le |\alpha_1-x|+|\alpha_2-x|\le 2|\alpha_2-x|,$$
and if $x \in I_2$
$$|\alpha_1-\alpha_2|\le 2|\alpha_1-x|.$$
Hence
$$|F(x)|=|L(x-\alpha_1)(x-\alpha_2)|\ge \frac{|L(x-\alpha_1)(\alpha_1-\alpha_2)|}{2}, \quad \text{if} \quad x\in I_1, $$
and
$$|F(x)| \ge \frac{|L(x-\alpha_2)(\alpha_1-\alpha_2)|}{2}, \quad \text{if} \quad  x \in I_2,$$
which gives the set inclusions
$$I_1 \cap \{ x \in \mathbb{R} : |F(x)|\le \varepsilon \} \subseteq \left \{ x \in \mathbb{R} : \frac{|L(x-\alpha_1)(\alpha_1-\alpha_2)|}{2} \le \varepsilon \right \},  $$
and
$$I_2 \cap \{ x \in \mathbb{R} : |F(x)|\le \varepsilon \} \subseteq \left \{ x \in \mathbb{R} : \frac{|L(x-\alpha_2)(\alpha_1-\alpha_2)|}{2} \le \varepsilon \right \}.  $$
Since
$$I_1 \cup I_2=\mathbb{R},$$
we get
\begin{align*}
\mu \left( \{ x \in \mathbb{R} : |F(x)|\le \varepsilon \} \right)&\le \mu \left( \left \{ x \in \mathbb{R} : \frac{|L(x-\alpha_1)(\alpha_1-\alpha_2)|}{2} \le \varepsilon \right \} \right) \\ & \quad \quad \quad +\mu \left( \left \{ x \in \mathbb{R} : \frac{|L(x-\alpha_2)(\alpha_1-\alpha_2)|}{2} \le \varepsilon \right \} \right),
\end{align*}
and the result follows since
$$\mu \left( \left \{ x \in \mathbb{R} : \frac{|L(x-\alpha_i)(\alpha_1-\alpha_2)|}{2} \le \varepsilon \right \} \right)\le \frac{4 \varepsilon}{|L(\alpha_1-\alpha_2)|}, \quad i=1,2.$$
\end{proof}

\begin{lemma}
\label{int bound 1}
Let $A,B,V$ be real numbers satisfying $$0\le A<B, \quad  B \ll V, \quad V\ge 1,$$ and let $t\ge 1$. Let the integers $v_1,v_2,v_3,v_4$ satisfy $$V/2<v_i\le V, \quad (v_1-v_3)(v_1-v_4)(v_2-v_3)(v_2-v_4) \neq 0,$$
and let 
$$F(x)=\frac{(x+v_1)(x+v_2)}{(x+v_3)(x+v_4)},$$
and
$$\Delta= (v_1-v_3)(v_1-v_4)(v_2-v_3)(v_2-v_4). $$
Then if $$v_1+v_2 \neq v_3+v_4,$$ we have 
$$\int_{A}^{B}F(x)^{it}dx \ll \frac{V^2}{t^{1/2}|\Delta|^{1/4}},$$
and if $$v_1+v_2=v_3+v_4,$$ 
we have
$$\int_{A}^{B}F(x)^{it}dx \ll\frac{V^4}{t|(v_1-v_4)(v_2-v_4)|}.$$
\end{lemma}
\begin{proof}
Consider first when
$$v_1+v_2\neq v_3+v_4.$$
Let 
$$F(x)=\frac{(x+v_1)(x+v_2)}{(x+v_3)(x+v_4)},$$
and
\begin{align}
\label{LMN}
L&=(v_1+v_2)-(v_3+v_4), \nonumber \\
M&=v_3v_4-v_1v_2, \\
N&=(v_1+v_2)v_3v_4-(v_3+v_4)v_1v_2, \nonumber
\end{align}
so that
\begin{align}
\label{delta equation}
M^2-LN=\Delta=(v_1-v_3)(v_1-v_4)(v_2-v_3)(v_2-v_4),
\end{align}
and
\begin{equation}
\label{F derivative}
F'(x)= \frac{Lx^2+2Mx+N}{(x+v_3)^2(x+v_4)^2}.
\end{equation}
Since the discriminant of the polynomial occuring in the numerator of~\eqref{F derivative} is 
$$4\Delta=4(v_1-v_3)(v_1-v_4)(v_2-v_3)(v_2-v_4),$$ 
by the assumption that $\Delta \neq 0$ we see that the polynomial
$$f(x)=Lx^2+2Mx+N,$$
has two distinct roots $\alpha_1, \alpha_2$. Hence by Lemma~\ref{poly small}, for any fixed $\varepsilon>0$
\begin{align*}
\left|\int_{A}^{B}F(x)^{it}dx\right|&\le \left|\int_{\substack{A\le x \le B \\ |f(x)|\le \varepsilon}}F(x)^{it}dx\right|+\left|
\int_{\substack{A\le x \le B \\ |f(x)|> \varepsilon}}F(x)^{it}dx\right| \\
& \ll \frac{\varepsilon}{|L(\alpha_1-\alpha_2)|}+\left|\int_{\substack{A\le x \le B \\ |f(x)|> \varepsilon}}F(x)^{it}dx\right|. 
\end{align*}
Integrating the second integral by parts gives
\begin{align*}
\int_{\substack{A\le x \le B \\ |f(x)|> \varepsilon}}F(x)^{it}dx&=\int_{\substack{A\le x \le B \\ |f(x)|> \varepsilon}}\frac{F'(x)}{F'(x)}F(x)^{it}dx \\
&=\left [\frac{F(x)^{1+it}}{(1+it)F'(x)} \right]_{\substack{A\le x \le B \\ |f(x)|> \varepsilon}} -\frac{1}{1+it}\int_{\substack{A\le x \le B \\ |f(x)|> \varepsilon}}\frac{F''(x)}{F'(x)^{2}}F(x)^{1+it}dx.
\end{align*}
Since  $$V/2<v_i\le V, \quad 0\le A<B\ll V,$$ 
we have for  $A\le x \le B$ and $|f(x)|>\varepsilon$
\begin{align*}
|F(x)|&\ll \left|\frac{(x+V)}{(2x+V)}\right|^2\ll 1, \\
\frac{1}{|F'(x)|}&=\left|\frac{(x+v_3)^2(x+v_4)^2}{Lx^2+2Mx+N}\right|\gg \frac{V^4}{\varepsilon},
\end{align*}
so that 
\begin{align*}
\left|\int_{\substack{A\le x \le B \\ |f(x)|> \varepsilon}}F(x)^{it}dx \right|&\le \frac{V^4}{t\varepsilon}+\frac{1}{t}\int_{\substack{A\le x \le B \\ |f(x)|> \varepsilon}}\left|\frac{F''(x)}{F'(x)^{2}}\right|dx.
\end{align*}
For the last integral, since 
\begin{equation}
\label{F''}
F''(x)=\frac{\frac{d((Lx^2+2Mx+N))}{dx}(x+v_3)^2(x+v_4)^2-(Lx^2+2Mx+N)\frac{d(((x+v_3)(x+v_4))^2)}{dx}}{(x+v_3)^4(x+v_4)^4},
\end{equation}
and the polynomial occuring in the numerator of~\eqref{F''} has degree at most $5$, we see that for real $x$ the function
$$\frac{F''(x)}{F'(x)^2},$$
has at most $5$ sign changes. Hence we may break the integral
$$\int_{\substack{A\le x \le B \\ |f(x)|> \varepsilon}}\left|\frac{F''(x)}{F'(x)^{2}}\right|dx,$$
into $O(1)$ integrals of the form
$$\left|\int_{A_i}^{B_i}\frac{F''(x)}{F'(x)^2}dx\right|,$$
for some $A_i,B_i$ with 
$$|f(x)|\ge \varepsilon \quad \text{for} \quad A_i \le x \le B_i.$$
Since
$$\left|\int_{A_i}^{B_i}\frac{F''(x)}{F'(x)^2}dx\right|\le \frac{1}{|F'(A_i)|}+ \frac{1}{|F'(B_i)|} \ll \frac{V^4}{\varepsilon},$$
we get
$$\int_{\substack{A\le x \le B \\ |f(x)|> \varepsilon}}\left|\frac{F''(x)}{F'(x)^{2}}\right|dx \ll \frac{V^4}{\varepsilon},$$
so that
$$\left|\int_{A}^{B}F(x)^{it}dx\right| \ll \frac{\varepsilon}{|L(\alpha_1-\alpha_2)|} + \frac{V^4}{t\varepsilon}.$$
Since $\alpha_1$ and $\alpha_2$ are the roots of the polynomial
$$f(x)=Lx^2+2Mx+N,$$
we see that
\begin{align*}
|L(\alpha_1-\alpha_2)|&=2|M^2-LN|^{1/2} \\ &=2|(v_1-v_3)(v_1-4)(v_2-v_3)(v_2-v_4)|^{1/2} \\
&=2|\Delta|^{1/2},
\end{align*}
so that taking $$\varepsilon=\frac{|\Delta|^{1/4}V^2}{t^{1/2}},$$
gives
$$\left|\int_{A}^{B}F(x)^{it}dx\right| \ll \frac{V^2}{|\Delta|^{1/4}t^{1/2}}.$$
Next consider when 
$$v_1+v_2=v_3+v_4,$$
so that~\eqref{F derivative} becomes
\begin{equation}
\label{FFFFF}
F'(x)= \frac{(2x+v_1+v_2)(v_1-v_4)(v_2-v_4)}{(x+v_3)^2(x+v_4)^2},
\end{equation}
and we see that $F'(x)$ has one zero at $x=-(v_1+v_2)/2$. Since 
\begin{equation}
\label{A,B,v_i}
0\le A<B, \quad  (v_1+v_2)/2 \ge V/2,
\end{equation}
integrating by parts as above gives
\begin{align*}
\left|\int_{A}^{B}F(x)^{it}dx\right|\ll \frac{1}{t}\max_{A\le x \le B}\frac{1}{|F'(x)|},
\end{align*}
so that by~\eqref{FFFFF} and~\eqref{A,B,v_i} we have
\begin{align*}
\left|\int_{A}^{B}F(x)^{it}dx\right| &\ll \frac{1}{t}\max_{\substack{A\le x \le B}}\frac{1}{|F'(x)|}\ll  \frac{V^4}{t|(v_1-v_4)(v_2-v_4)|}.
\end{align*}
\end{proof}
\begin{lemma}
\label{int bound 2}
Let $A,B,V$ be real numbers satisfying $$0\le A<B, \quad  B \ll V,$$ and let $t\ge 1$. For integers $v_1, v_4$ satisfying $$V/2<v_1,v_4 \le V, \quad v_1\neq v_4,$$ we have
$$ \int_{A}^{B}\left(\frac{x+v_1}{x+v_4}\right)^{it}dx\ll \frac{V^2}{t|v_1-v_4|}.$$
\end{lemma}
\begin{proof}
Let 
$$F(x)=\frac{x+v_1}{x+v_4},$$
so that
$$F'(x)=\frac{v_4-v_1}{(x+v_4)^2}.$$
Integrating by parts as in the proof of Lemma~\ref{int bound 1} gives
\begin{align*}
\left|\int_{A}^{B}\left(\frac{x+v_1}{x+v_4}\right)^{it}\right| &\ll \max_{A\le x \le B}\frac{1}{t|F'(x)|} \\
&\ll\frac{V^2}{t|v_1-v_4|}.
\end{align*}
\end{proof}
The following is due to Burgess~\cite[Lemma~2,3,4]{Burg1}.
\begin{lemma}
\label{bur}
Let $p$ be prime and $\alpha$ be an integer. For integers $v_1,v_2,v_3,v_4$ let $N(p^{\alpha})$ denote the number of solutions to the congruence
$$((v_1+v_2)-(v_3+v_4))x^2+(v_3v_4-v_1v_2)x+(v_1+v_2)v_3v_4-(v_3+v_4)v_1v_2\equiv 0 \pmod {p^{\alpha}},$$
with $1\le x \le p^{\alpha}$ and let $\chi$ be a primitive character$\pmod{ p^{\alpha}}$. Then if $p$ is odd
\begin{align*}
\left|\sum_{n=1}^{p^{\alpha}}\chi\left(\frac{(x+v_1)(x+v_2)}{(x+v_3)(x+v_4)}\right)\right|\le \begin{cases} N(p^{\alpha/2})p^{\alpha/2}, \quad \text{if $\alpha$ is even}, \\ 
N(p^{(\alpha-1)/2})p^{\alpha/2}+N(p^{(\alpha+1)/2})p^{(\alpha-1)/2}, \\ \quad \quad \quad \quad \quad  \quad \ \  \text{if $\alpha$ is odd,}
 \end{cases}
\end{align*}
and if $p=2$
\begin{align*}
\left|\sum_{n=1}^{2^{\alpha}}\chi\left(\frac{(x+v_1)(x+v_2)}{(x+v_3)(x+v_4)}\right)\right|\le \begin{cases} N(2^{\alpha/2})2^{\alpha/2}, \quad \text{if $\alpha$ is even}, \\ 
N(2^{(\alpha+1)/2})2^{(\alpha+1)/2}, \quad \text{if $\alpha$ is odd.}
 \end{cases}
\end{align*}
\end{lemma}
The following Lemma is due to Huxley~\cite{Hu}, see~\cite[Section 3]{St} for related results.
\begin{lemma}
\label{sandor}
Let $F(x)\in \mathbb{Z}[x]$ be a  polynomial of degree $r\ge2$ and let $\Delta$ denote the discriminant of $F$.  For prime $p$ and integer $\alpha$, let  $N(F,p^{\alpha})$ denote the number of solutions to the congruence
$$F(x)\equiv 0 \pmod {p^{\alpha}}, \quad 1\le x \le p^{\alpha}.$$
Then if $\Delta \neq 0$ we have
$$N(F,p^{\alpha})\le r(p^{\alpha},\Delta)^{1/2}.$$
\end{lemma} 
Using the proof of~\cite[Lemma~7]{Burg1} with Lemma~\ref{bur} and Lemma~\ref{sandor} gives,
\begin{lemma}
\label{Burgess}
For integer $q$, let $\chi$ be a primitive character$\pmod q$ and suppose the integers $v_1,v_2,v_3,v_4$  satisfy
$$\Delta=(v_1-v_3)(v_1-v_4)(v_2-v_3)(v_2-v_4)\neq 0.$$
 Then if 
$$v_1+v_2\neq v_3+v_4,$$
we have
$$\left|\sum_{\lambda=1}^{q}\chi \left(\frac{(\lambda+v_1)(\lambda+v_2)}{(\lambda+v_3)(\lambda+v_4)}\right)\right|\le (q,\Delta)^{1/2}q^{1/2+o(1)},$$
and if 
$$v_1+v_2=v_3+v_4,$$
we have
$$\left|\sum_{\lambda=1}^{q}\chi \left(\frac{(\lambda+v_1)(\lambda+v_2)}{(\lambda+v_3)(\lambda+v_4)}\right)\right|\le (q,(v_1-v_4)(v_2-v_4))q^{1/2+o(1)}.$$
\end{lemma}
\begin{proof}
Consider first when
\begin{equation}
\label{assumption 1  L9}
v_1+v_2\neq v_3+v_4,
\end{equation}
then if $\chi$ is a primitive character$\pmod {p^{\alpha}},$ as in~\cite[Lemma~5]{Burg1}, since the discriminant of the polynomial
\begin{equation}
\label{polyonmial congruence bound}
((v_1+v_2)-(v_3+v_4))x^2+(v_3v_4-v_1v_2)x+(v_1+v_2)v_3v_4-(v_3+v_4)v_1v_2,
\end{equation}
is
$$4\Delta=4(v_1-v_3)(v_1-v_4)(v_2-v_3)(v_2-v_4),$$
we see from Lemma~\ref{bur} and Lemma~\ref{sandor} that
\begin{equation}
\label{b bound 1}
\left|\sum_{n=1}^{p^{\alpha}}\chi\left(\frac{(x+v_1)(x+v_2)}{(x+v_3)(x+v_4)}\right)\right|\ll (p^{\alpha},\Delta)^{1/2} p^{\alpha/2},
\end{equation}
which under the assumption~\eqref{assumption 1 L9} gives the desired result when $q$ is a prime power. For the general case, suppose $\chi$ is a primitive character$\pmod q$ and let $q=p_1^{\alpha_1}\dots p_k^{\alpha_k}$ be the prime factorization of $q$. By the Chinese remainder theorem, there exists  $\chi_1,\dots,\chi_k$, where each $\chi_i$ is a primitive character$\pmod {p_i^{\alpha_i}}$ such that
$$\chi=\chi_1\dots\chi_k.$$
Let 
$$F(x)=\frac{(x+v_1)(x+v_2)}{(x+v_3)(x+v_4)},$$
and writing $q_i=q/p_i^{\alpha_i}$ we have
\begin{align}
\label{CRT}
\sum_{n=1}^{p^{\alpha}}\chi \left(F(n)\right)&=\sum_{n_1=1}^{p_1^{\alpha_1}}\dots \nonumber \sum_{n_k=1}^{p_k^{\alpha_k}}\chi_1(F(n_1q_1+\dots n_kq_k))\dots \chi_k(F(n_1q_1+\dots n_kq_k)) \\
&=\prod_{i=1}^{k}\left(\sum_{n_i=1}^{p_i^{\alpha_i}}\chi_i(F(n_iq_i))  \right)=\prod_{i=1}^{k}\left(\sum_{n_i=1}^{p_i^{\alpha_i}}\chi_i(F(n_i))\right).
\end{align}
Letting $\omega(q)$ denote the number of distinct prime factors of $q$, by~\eqref{b bound 1} and~\eqref{CRT} we have for some absolute constant $C$,
$$\left|\sum_{n=1}^{p^{\alpha}}\chi \left(F(n)\right)\right|\le C^{\omega(q)}(q,\Delta)^{1/2}q^{1/2}\le (q,\Delta)^{1/2}q^{1/2+o(1)}.$$
Next suppose
$$v_1+v_2=v_3+v_4,$$
so that~\eqref{polyonmial congruence bound} becomes
$$(2x+v_1+v_2)(v_1-v_4)(v_2-v_4).$$
For $\chi$ be a primitive character$\pmod {p^{\alpha}}$, since the number of solutions to the congruence
$$(2x+v_1+v_2)(v_1-v_4)(v_2-v_4)\equiv 0 \pmod {p^{\alpha}}, \quad 1\le x \le p^{\alpha},$$
is bounded by 
$$(p^{\alpha},2(v_4-v_1)(v_4-v_2))\le 2(p^{\alpha},(v_1-v_4)(v_2-v_4)),$$
 we have from  Lemma~\ref{bur}
$$\left|\sum_{n=1}^{p^{\alpha}}\chi\left(\frac{(x+v_1)(x+v_2)}{(x+v_3)(x+v_4)}\right)\right|\ll (p^{\alpha},(v_1-v_4)(v_2-v_4)) p^{\alpha/2},$$
so that using the Chinese remainder theorem as above gives
$$\left|\sum_{n=1}^{p^{\alpha}}\chi \left(F(n)\right)\right|\le (q,(v_1-v_4)(v_2-v_4))q^{1/2+o(1)}. $$
\end{proof}
\begin{lemma}
\label{double mean value}
For integer $q$, let  $\chi$ be a primitive character$\pmod q$ and let $A,B,V$ be real numbers satisfying $$0\le A<B, \quad  B \ll V,$$ and let $t\ge 1$. For any real number $\alpha$ we have
\begin{align*}
& \int_{A}^{B}\sum_{\lambda=1}^{q}\max_{V/2<Q\le V}\left|\sum_{V/2<v\le Q}\chi(\lambda+v)(x+v)^{it}e^{2\pi i \alpha v} \right|^4dx \le \\ & \quad \quad \quad \quad \left(qV^3+\frac{q^{1/2}V^{5}}{t^{1/2}}\right)(qV)^{o(1)}.
\end{align*}
\end{lemma}
\begin{proof}
We first show that for $C\le V/2$
\begin{align}
\label{remove max}
& \int_{A}^{B}\sum_{\lambda=1}^{q}\left|\sum_{V/2<v\le V/2+C}\chi(\lambda+v)(x+v)^{it}e^{2\pi i \alpha v} \right|^4dx\le  \\ & \quad \quad \quad \quad  C\left(qV^3+\frac{q^{1/2}V^{5}}{t^{1/2}}\right)(qV)^{o(1)}, \nonumber
\end{align}
 then we complete the proof using the argument of Heath-Brown~\cite[Section~5]{HB2}.
 Expanding the inner sum in~\eqref{remove max}
\begin{align*}
& \int_{A}^{B}\sum_{\lambda=1}^{q}\left|\sum_{V/2<v\le V/2+C}\chi(\lambda+v)(x+v)^{it}e^{2\pi i \alpha v} \right|^4dx \le
\\ &
\sum_{V/2<v_1,v_2,v_3,v_4 \le V/2+C}\left|\int_{A}^{B}\left(\frac{(x+v_1)(x+v_2)}{(x+v_3)(x+v_4)} \right)^{it}dx \right|\left|\sum_{\lambda=1}^{q}\chi \left(\frac{(\lambda+v_1)(\lambda+v_2)}{(\lambda+v_3)(\lambda+v_4)} \right)  \right|.
\end{align*}
We break the outer summation over $(v_1,v_2,v_3, v_4)$ into two sets. In the first set $\cV_1$, we put all $(v_1,v_2,v_3,v_4)$ that contain at most two distinct integers and in the second set $\cV_2$, we put the remaining $(v_1,v_2,v_3,v_4)$. Estimating the inner sum and integral trivially for the first set, since $$C\le V/2, \quad 0\le A<B\ll V,$$ we get
\begin{align*}
& \int_{A}^{B}\sum_{\lambda=1}^{q}\left|\sum_{V/2<v\le V/2+C}\chi(\lambda+v)(x+v)^{it} \right|^4dx \ll \\
& \  CqV^2+\sum_{(v_1,v_2,_3,v_4)\in \cV_2}\left|\int_{0}^{B}\left(\frac{(x+v_1)(x+v_2)}{(x+v_3)(x+v_4)} \right)^{it}dx \right|\left|\sum_{\lambda=1}^{q}\chi \left(\frac{(\lambda+v_1)(\lambda+v_2)}{(\lambda+v_3)(\lambda+v_4)} \right)  \right|.
\end{align*}
Writing
$$\Delta=(v_1-v_3)(v_1-v_4)(v_2-v_3)(v_2-v_4),$$
we split $\cV_2$ into three sets,
\begin{align*}
\cV_3&=\{ \  (v_1,v_2,v_3,v_4)\in \cV_2 : \ \Delta=0 \  \}, \\
\cV_4&=\{ \  (v_1,v_2,v_3,v_4)\in \cV_2 : \ \Delta \neq 0, \  v_1+v_2\neq v_3+v_4 \  \}, \\
\cV_5&=\{ \  (v_1,v_2,v_3,v_4)\in \cV_2 : \ \Delta \neq 0,\  v_1+v_2=v_3+v_4 \  \}.
\end{align*}
By Lemma~\ref{int bound 1} and Lemma~\ref{Burgess} we have
\begin{align*}
&\sum_{(v_1,v_2,_3,v_4)\in \cV_4}\left|\int_{A}^{B}\left(\frac{(x+v_1)(x+v_2)}{(x+v_3)(x+v_4)} \right)^{it}dx \right|\left|\sum_{\lambda=1}^{q}\chi \left(\frac{(\lambda+v_1)(\lambda+v_2)}{(\lambda+v_3)(\lambda+v_4)} \right)  \right| \ll \\
& \ \ \ \ \frac{q^{1/2+o(1)}V^2}{t^{1/2}}\sum_{(v_1,v_2,_3,v_4)\in \cV_4}\frac{(q,\Delta)^{1/2}}{|\Delta|^{1/4}}.
\end{align*}
Since
\begin{align*}
&\sum_{(v_1,v_2,_3,v_4)\in \cV_4}\frac{(q,\Delta)^{1/2}}{|\Delta|^{1/4}}=\sum_{(v_1,v_2,_3,v_4)\in \cV_4}\frac{(q,(v_1-v_3)(v_1-v_4)(v_2-v_3)(v_2-v_4))^{1/2}}{|(v_1-v_3)(v_1-v_4)(v_2-v_3)(v_2-v_4)|^{1/4}} \\
 &  \quad \quad \quad  \le \sum_{(v_1,v_2,_3,v_4)\in \cV_4}\frac{(q,(v_1-v_3))^{1/2}(q,(v_1-v_4))^{1/2}(q,(v_2-v_3))^{1/2}(q,(v_2-v_4))^{1/2}}{|(v_1-v_3)(v_1-v_4)(v_2-v_3)(v_2-v_4)|^{1/4}},
\end{align*}
we break the above sum into $q^{o(1)}$ sums of the form
\begin{align*}
\sum_{\substack{(v_1,v_2,_3,v_4)\in \cV_4 \\ (q,v_i-v_j)=d_{i,j} \\ i=1,2,  j=3,4}}\frac{(q,(v_1-v_3))^{1/2}(q,(v_1-v_4))^{1/2}(q,(v_2-v_3))^{1/2}(q,(v_2-v_4))^{1/2}}{|(v_1-v_3)(v_1-v_4)(v_2-v_3)(v_2-v_4)|^{1/4}},
\end{align*}
where each $d_{i,j}$ is a divisor of $q$. Since
$$\frac{(q,v_1-v_3)^{1/2}(q,(v_1-v_4))^{1/2}}{|(v_1-v_3)(v_1-v_4)|^{1/4}}\le\frac{(q,v_1-v_3)}{|(v_1-v_3)|^{1/2}}+
\frac{(q,v_1-v_4)}{|(v_1-v_4)|^{1/2}},
$$
we have
\begin{align*}
&\sum_{\substack{(v_1,v_2,_3,v_4)\in \cV_4 \\ (q,v_i-v_j)=d_{i,j} \\ i=1,2,  j=3,4}}\frac{(q,(v_1-v_3))^{1/2}(q,(v_1-v_4))^{1/2}(q,(v_2-v_3))^{1/2}(q,(v_2-v_4))^{1/2}}{|(v_1-v_3)(v_1-v_4)(v_2-v_3)(v_2-v_4)|^{1/4}} \\
& \quad \quad \quad \le 
\sum_{\substack{(v_1,v_2,_3,v_4)\in \cV_4 \\ (q,v_i-v_j)=d_{i,j} \\ i=1,2,  j=3,4}}\frac{d_{1,3}(d_{2,4}d_{2,3})^{1/2}}{|(v_1-v_3)|^{1/2}|(v_2-v_3)(v_2-v_4)|^{1/4}} \\
& \quad \quad \quad \quad \quad  +\sum_{\substack{(v_1,v_2,_3,v_4)\in \cV_4 \\ (q,v_i-v_j)=d_{i,j} \\ i=1,2,  j=3,4}}\frac{d_{1,4}(d_{2,4}d_{2,3})^{1/2}}{|(v_1-v_4)|^{1/2}|(v_2-v_3)(v_2-v_4)|^{1/4}}.
\end{align*}
Considering the first sum, for integers $M_1,M_2,M_3$ since the number of solutions to the equations
\begin{align*}
v_1-v_3&=d_{1,3}M_1, \\
v_2-v_3&=d_{2,3}M_2,\\
v_2-v_4&=d_{2,4}M_3, 
\end{align*}
with $(v_1,v_2,v_3,v_4) \in \cV_4$ is bounded by $C$, we have
\begin{align*}
&\sum_{\substack{(v_1,v_2,_3,v_4)\in \cV_4 \\ (q,v_i-v_j)=d_{i,j} \\ i=1,2,  j=3,4}}\frac{d_{1,3}(d_{2,4}d_{2,3})^{1/2}}{|(v_1-v_3)|^{1/2}|(v_2-v_3)(v_2-v_4)|^{1/4}}\\ & \quad \quad \quad \quad \quad \quad \quad \ll \left(d_{1,3}(d_{2,4}d_{2,3})^{1/2}\sum_{\substack{M_1\le C/d_{1,3} \\ M_2\le C/d_{2,3} \\ M_3\le C/d_{2,4}}}\frac{1}{(d_{1,3}M_1)^{1/2}(d_{2,3}M_2d_{2,4}M_3)^{1/4}}\right)C
\\ & \quad \quad \quad \quad \quad \quad \quad = \left(d_{1,3}^{1/2}(d_{2,4}d_{2,3})^{1/4}\sum_{\substack{M_1\le C/d_{1,3} \\ M_2\le C/d_{2,3} \\ M_3\le C/d_{2,4}}}\frac{1}{M_1^{1/2}(M_2M_3)^{1/4}}\right)C \\ & \quad \quad \quad \quad \quad \quad \quad  \ll C^3.
\end{align*}
A similar argument shows
$$\sum_{\substack{(v_1,v_2,_3,v_4)\in \cV_4 \\ (q,v_i-v_j)=d_{i,j} \\ i=1,2,  j=3,4}}\frac{d_{1,4}(d_{2,4}d_{2,3})^{1/2}}{|(v_1-v_4)|^{1/2}|(v_2-v_3)(v_2-v_4)|^{1/4}}\ll C^3,$$
and since $C\le V$, we get
\begin{align*}
& \sum_{(v_1,v_2,_3,v_4)\in \cV_4}\left|\int_{A}^{B}\left(\frac{(x+v_1)(x+v_2)}{(x+v_3)(x+v_4)} \right)^{it}dx \right|\left|\sum_{\lambda=1}^{q}\chi \left(\frac{(\lambda+v_1)(\lambda+v_2)}{(\lambda+v_3)(\lambda+v_4)} \right)  \right| \\
& \quad \quad \quad \quad \quad \quad \le\frac{Cq^{1/2+o(1)}V^4}{t^{1/2}}. 
\end{align*}
For summation over the set $\cV_5$, we have by Lemma~\ref{int bound 1} and Lemma~\ref{Burgess}
\begin{align*}
& \sum_{(v_1,v_2,_3,v_4)\in \cV_5}\left|\int_{A}^{B}\left(\frac{(x+v_1)(x+v_2)}{(x+v_3)(x+v_4)} \right)^{it}dx \right|\left|\sum_{\lambda=1}^{q}\chi \left(\frac{(\lambda+v_1)(\lambda+v_2)}{(\lambda+v_3)(\lambda+v_4)} \right)  \right|\ll \\ & \quad \quad 
\frac{V^4q^{1/2+o(1)}}{t} \sum_{(v_1,v_2,_3,v_4)\in \cV_5}\frac{(q,v_1-v_4)(q,v_2-v_4)}{|(v_1-v_4)(v_2-v_4)|},
\end{align*}
and for integers $M_1,M_2$ since the number of solutions to the equations
\begin{align*}
v_1-v_4&=M_1, \\
v_2-v_4&=M_2, \\
v_1+v_2&=v_3+v_4, 
\end{align*}
with $(v_1,v_2,v_3,v_4)\in \cV_5$ is bounded by $C$, we get
\begin{align*}
 \sum_{(v_1,v_2,_3,v_4)\in \cV_5}\frac{(q,v_1-v_4)(q,v_2-v_4)}{|(v_1-v_4)(v_2-v_4)|}&\ll C\sum_{M_1,M_2 \le C}\frac{(q,M_1)(q,M_2)}{M_1M_2}\\ &\ll C^{1+o(1)}q^{o(1)},
\end{align*}
so that
\begin{align*}
&\sum_{(v_1,v_2,_3,v_4)\in \cV_5}\left|\int_{A}^{B}\left(\frac{(x+v_1)(x+v_2)}{(x+v_3)(x+v_4)} \right)^{it}dx \right|\left|\sum_{\lambda=1}^{q}\chi \left(\frac{(\lambda+v_1)(\lambda+v_2)}{(\lambda+v_3)(\lambda+v_4)} \right)  \right|\ll \\ 
& \quad \quad \quad \quad \frac{Cq^{1/2+o(1)}V^{4+o(1)}}{t}.
\end{align*}
Considering $\cV_3$, since
$$(v_1-v_3)(v_1-v_4)(v_2-v_3)(v_2-v_4)=0,$$
we have by symmetry,
 \begin{align*}
&\sum_{(v_1,v_2,_3,v_4)\in \cV_3}\left|\int_{A}^{B}\left(\frac{(x+v_1)(x+v_2)}{(x+v_3)(x+v_4)} \right)^{it}dx \right|\left|\sum_{\lambda=1}^{q}\chi \left(\frac{(\lambda+v_1)(\lambda+v_2)}{(\lambda+v_3)(\lambda+v_4)} \right)  \right| \\ & \ll
\quad V^2\sum_{\substack{V/2<v_1,v_4 \le V/2+C \\ v_4<v_1}}\left|\int_{A}^{B}\left(\frac{x+v_1}{x+v_4} \right)^{it}dx \right|\left|\sum_{\lambda=1}^{q}\chi \left(\frac{\lambda+v_1}{\lambda+v_4} \right)  \right|.
\end{align*}
Using~\cite[Equation~3.5]{IwKow} and~\cite[Equation~12.51]{IwKow} we have
\begin{align*}
\sum_{\lambda=1}^{q}\chi \left(\frac{\lambda+v_1}{\lambda+v_4} \right)=\sum_{\substack{\lambda=1 \\ (\lambda,q)=1}}^{q}e^{2\pi i (v_1-v_2)\lambda/q} \ll (q,v_1-v_2),
\end{align*}
so that by Lemma~\ref{int bound 2} 
\begin{align*}
&\sum_{\substack{V/2<v_1,v_4 \le V_/2+C \\ v_4<v_1}}\left|\int_{A}^{B}\left(\frac{x+v_1}{x+v_4} \right)^{it}dx \right|\left|\sum_{\lambda=1}^{q}\chi \left(\frac{\lambda+v_1}{\lambda+v_4} \right)  \right| \\ & \quad \quad \quad \quad \quad \ll \frac{V^2}{t}\sum_{v_4<v_1\le C}\frac{(q,v_1-v_4)}{v_1-v_4} \le \frac{Cq^{o(1)}V^{2+o(1)}}{t},
\end{align*}
hence we get
 \begin{align*}
& \sum_{(v_1,v_2,_3,v_4)\in \cV_3}\left|\int_{A}^{B}\left(\frac{(x+v_1)(x+v_2)}{(x+v_3)(x+v_4)} \right)^{it}dx \right|\left|\sum_{\lambda=1}^{q}\chi \left(\frac{(\lambda+v_1)(\lambda+v_2)}{(\lambda+v_3)(\lambda+v_4)} \right)  \right|\\ & \quad \quad \quad \quad \quad \quad \ll  \frac{Cq^{o(1)}V^{4+o(1)}}{t}.
\end{align*}
Combining the estimates for $\cV_1,\cV_3,\cV_4,\cV_5$ gives
\begin{align}
\label{BBBBB}
&\int_{A}^{B}\sum_{\lambda=1}^{q}\left|\sum_{V/2<v\le V/2+C}\chi(\lambda+v)(x+v)^{it}e^{2\pi i \alpha v} \right|^4dx\le \nonumber \\ &
\quad \quad \quad \quad  C\left(qV^2+\frac{q^{1/2+o(1)}V^{4+o(1)}}{t^{1/2}}\right)(qV)^{o(1)}.
\end{align}
Next we use~\eqref{BBBBB} as in the argument of~\cite[Section~5]{HB2} to bound
$$\int_{A}^{B}\sum_{\lambda=1}^{q}\max_{V/2<Q\le V}\left|\sum_{V/2<v\le Q}\chi(\lambda+v)(x+v)^{it}e^{2\pi i \alpha v} \right|^4dx.$$
For each $1\le \lambda \le q$ and $A\le x \le B$,  let $Q_{\lambda,x}$ be the integer defined by 
\begin{align*}
& \max_{V/2<Q\le V}\left|\sum_{V/2<v\le Q}\chi(\lambda+v)(x+v)^{it}e^{2\pi i \alpha v} \right|^4= \\ &
\quad \quad \quad \quad 
\left|\sum_{V/2<v\le V/2+ Q_{\lambda,x}}\chi(\lambda+v)(x+v)^{it}e^{2\pi i \alpha v} \right|^4,
\end{align*}
and let 
$$Q_{\lambda,x}=\sum_{r \le R}\delta_{\lambda,x}(r)2^r,$$
be the binary expansion of $Q_{\lambda,x}$, where $R$ is the largest integer such that $2^R\le V$ and we define
$$s_{\lambda,x}(r)=\sum_{r<t\le R}\delta_{\lambda,x}(t)2^{t-r}.$$
 Then writing
$$H_{\lambda,x}(C,D)=\sum_{D<v\le D+C}\chi(\lambda+v)(x+v)^{it}e^{2\pi i \alpha v},$$
we have as in~\cite{HB2}
$$H_{\lambda,x}(Q_{\lambda,x},V/2)=\sum_{r\le R}\delta_{\lambda,x}(r)H_{\lambda,x}(2^r,V/2+s_{\lambda,x}(r)2^r).$$
By H\"{o}lder's inequality
\begin{align*}
|H_{\lambda,x}(Q(\lambda,x),V/2)|^4&\le R^3\left(\sum_{r\le R}\delta_{\lambda,x}(r)|H_{\lambda,x}(2^r,V/2+s_{\lambda,x}(r)2^r)|^4\right) \\
&\le V^{o(1)}\sum_{r\le R}|H_{\lambda,x}(2^r,V/2+s_{\lambda,x}(r)2^r)|^4,
\end{align*}
and since  $s(r,\lambda,x)\le 2^{R-r}$, we have
$$|H_{\lambda,x}(Q(\lambda,x),V/2)|^4\le V^{o(1)}\sum_{r\le R}\sum_{s\le 2^{R-r}}|H_{\lambda,x}(2^r,V/2+s2^r)|^4.$$
Hence by~\eqref{BBBBB}
\begin{align*}
& \int_{A}^{B}\sum_{\lambda=1}^{q}\max_{V/2<Q\le V}\left|\sum_{V/2<v\le Q}\chi(\lambda+v)(x+v)^{it}e^{2\pi i \alpha v} \right|^4dx\le \\
& \quad \quad \quad  q^{o(1)}\sum_{r\le R}\sum_{s\le 2^{R-r}}\int_{A}^{B}\sum_{\lambda=1}^{q}|H_{\lambda,x}(2^r,V/2+s2^r)|^4dx \le \\ & \quad \quad \quad \quad  \left(qV^2+\frac{q^{1/2+o(1)}V^{4+o(1)}}{t^{1/2}}\right)(qV)^{o(1)}\sum_{r\le R}\sum_{s\le 2^{R-r}}2^r,
\end{align*}
so that
$$\sum_{r\le R}\sum_{s\le 2^{R-r}}2^r\le \sum_{r\le R}2^R\le 2^RV^{o(1)},$$
which gives
\begin{align*}
& \int_{A}^{B}\sum_{\lambda=1}^{q}\max_{V/2<Q\le V}\left|\sum_{V/2<v\le Q}\chi(\lambda+v)(x+v)^{it}e^{2\pi i \alpha v} \right|^4dx\le \\
& \quad \quad \quad \left(qV^3+\frac{q^{1/2+o(1)}V^{5+o(1)}}{t^{1/2}}\right)(qV)^{o(1)}.
\end{align*}
\end{proof}
\section{Proof of Theorem~\ref{main1}}
We begin with some ideas from the proof of~\cite[Theorem~1]{FI}. Let 
$$f(x)=\begin{cases}\min(x-M,1,M+N-x), \quad \text{if} \quad M\le x \le M+N, \\ 
0, \quad \quad \quad \quad \quad \quad \quad \quad \quad \quad \quad \quad \    \text{otherwise}, \end{cases}$$
so that $f(x)$ is a continuous function equal to $1$ for integers $M<n\le M+N$ and $0$ otherwise, hence 
$$\sum_{M<n\le M+N}\chi(n)n^{it}=\sum_{M-N<n\le M+N}f(n)\chi(n)n^{it}.$$
We define the integers 
\begin{equation}
\label{U,V}
U= \lfloor Nt^{-1/4}q^{-1/4} \rfloor, \quad V=\lfloor t^{1/4}q^{1/4} \rfloor,
\end{equation}
and the sets 
$$\cU=\{ \  U/2<u \le U, \ \ \  (u,q)=1 \  \}, \quad \cV=\{ \  V/2<v\le V \ \},$$
so by assumption on $N$ and $t$ we have $U, V\ge 1$ and $UV\le N$. Since $$ N \ge UV,$$ we have
$$\sum_{M<n\le M+N}\chi(n)n^{it}=\frac{1}{\# \cU \#\cV}\sum_{M-N<n\le M+N}\sum_{u\in \cU}\sum_{v\in \cV}f(n+uv)\chi(n+uv)(n+uv)^{it}.$$
As in~\cite{FI}, let $g(y)$ denote the Fourier transform of $f(x)$, so that 
$$f(x)=\int_{-\infty}^{\infty}g(y)e^{-2\pi i xy}dy,$$
and
\begin{align*}
&\sum_{M<n\le M+N}\chi(n)n^{it}= \\  & \frac{1}{\# \cU \#\cV}\sum_{M-N<n\le M+N}\sum_{u\in \cU}\int_{-\infty}^{\infty}g(y)\left(\sum_{v\in \cV}\chi(n+uv)(n+uv)^{it}e^{2\pi i (n+uv)y}\right)dy.
\end{align*}
The change of variable $x=uy$ in the above integral gives
\begin{align*}
& \left|\sum_{M<n\le M+N}\chi(n)n^{it} \right|\le \\ &  \frac{1}{\# \cU \#\cV}\sum_{M-N<n\le M+N}\sum_{u\in \cU}\int_{-\infty}^{\infty}\frac{1}{u}\left|g\left(\frac{y}{u}\right)\right|\left|\sum_{v\in \cV}\chi(nu^{*}+v)(nu^{-1}+v)^{it}e^{2\pi i vy}\right|dy,
\end{align*}
where $u^{*}$ denotes the multiplicative inverse of $u \pmod q$.
As in~\cite[Theorem~2]{FI}, we have
$$\frac{1}{u}\left|g\left(\frac{y}{u}\right)\right|\ll \min \left(N,\frac{1}{|y|},\frac{U}{|y|^2}\right),$$
so that
\begin{align*}
& \left|\sum_{M<n\le M+N}\chi(n)n^{it} \right|\le \\ &  \frac{1}{\# \cU \#\cV}\int_{-\infty}^{\infty}\min \left(N,\frac{1}{|y|},\frac{U}{|y|^2}\right) \times \\ & \quad \quad \quad \sum_{M-N<n\le M+N}\sum_{u\in \cU}\left|\sum_{v\in \cV}\chi(nu^{*}+v)(nu^{-1}+v)^{it}e^{2\pi i vy}\right|dy.
\end{align*}
Let $\alpha$ be defined by
\begin{align*}
&\max_{y \in \mathbb{R}}\sum_{M-N<n\le M+N}\sum_{u\in \cU}\left|\sum_{v\in \cV}\chi(n+uv)(n+uv)^{it}e^{2\pi i vy}\right| \\ & \quad \quad \quad \quad \quad =\sum_{M-N<n\le M+N}\sum_{u\in \cU}\left|\sum_{v\in \cV}\chi(n+uv)(n+uv)^{it}e^{2\pi i v\alpha}\right|,
\end{align*}
so that
\begin{align*}
& \left|\sum_{M<n\le M+N}\chi(n)n^{it} \right|\le \\ &  \left(\int_{-\infty}^{\infty}\min \left(N,\frac{1}{|y|},\frac{U}{|y|^2}\right)dy\right) \times \\ & \quad \quad \quad   \frac{1}{\# \cU \#\cV}\sum_{M-N<n\le M+N}\sum_{u\in \cU}\left|\sum_{v\in \cV}\chi(nu^{*}+v)(nu^{-1}+v)^{it}e^{2\pi i v\alpha}\right|.
\end{align*}
Since 
$$\int_{-\infty}^{\infty}\min \left(N,\frac{1}{|y|},\frac{U}{|y|^2}\right)dy \le (qt)^{o(1)},$$
we get
\begin{align*}
& \left|\sum_{M<n\le M+N}\chi(n)n^{it} \right|\le \\ &  \frac{(qt)^{o(1)}}{\# \cU \#\cV}\sum_{M-N<n\le M+N}\sum_{u\in \cU}\left|\sum_{v\in \cV}\chi(nu^{*}+v)(nu^{-1}+v)^{it}e^{2\pi i v\alpha}\right|.
\end{align*}
Let 
\begin{equation}
\label{W def}
W=\sum_{M-N<n\le M+N}\sum_{u\in \cU}\left|\sum_{v\in \cV}\chi(nu^{*}+v)(nu^{-1}+v)^{it}e^{2\pi i v\alpha}\right|,
\end{equation}
and
\begin{equation}
\label{H}
H= tV^{-1}. 
\end{equation}
For integer $h$, we consider the intervals
$$I_{h}=\left [ \frac{h}{H}, \frac{h+1}{H} \right),$$
and define the sets $\Omega_{h}$ by
$$\Omega_{h}=\{ (n,u) : M-N<n\le M+N, \ \ u\in \cU, \ \  \frac{n}{u}\in I_h \  \},$$
so that by assumption on $M$, $\Omega_h$ is empty for $h<0$ and $h>6HN/U$. Using ideas from~\cite{HB2}, let 
$$F(n,u,x)=\left(\sum_{V/2<v\le V}\chi(nu^{*}+v)(x+v)^{it}e^{2\pi i v\alpha}\right)^4,$$
so that we may rewrite~\eqref{W def} as 
\begin{equation}
\label{W def1}
W=\sum_{M-N<n\le M+N}\sum_{u\in \cU}|F(n,u,nu^{-1}|^{1/4}.
\end{equation}
If $nu^{-1}\in \Omega_h$, by Lemma~\ref{sg} we have
\begin{align}
\label{sga}
|F(n,u,nu^{-1})|\le H\int_{h/H}^{(h+1)/H}|F(n,u,x)|dx+ \int_{h/H}^{(h+1)/H}|F'(n,u,x)|dx.
\end{align}
Let $I_{h}(\lambda)$ denote the number of solutions to the congruence
$$nu^{*}\equiv \lambda \pmod q, \quad (n,u)\in \Omega_h,$$
and writing
\begin{equation}
\label{G def def}
G(\lambda,x)= \sum_{V/2<v\le V}\chi(\lambda+v)(x+v)^{it}e^{2\pi i v\alpha},
\end{equation}
 we have by~\eqref{W def1} and~\eqref{sga},
\begin{align}
\label{W and G} 
\nonumber &|W|\ll \sum_{h\le 6HN/U}\sum_{\lambda=1}^{q}I_h(\lambda)\left(H\int_{h/H}^{(h+1)/H}\left|G(\lambda,x)\right|^4dx \right)^{1/4} \\ & \quad \quad \quad + \sum_{h\le 6HN/U}\sum_{\lambda=1}^{q}I_h(\lambda)\left(\int_{h/H}^{(h+1)/H}\left|G'(\lambda,x)\right|\left|G(\lambda,x)\right|^3dx \right)^{1/4}.
\end{align}
By H\"{o}lder's inequality
\begin{align*}
\int_{h/H}^{(h+1)/H}\left|G'(\lambda,x)\right|\left|G(\lambda,x)\right|^3dx &\le \left(\int_{h/H}^{(h+1)/H}\left|G(\lambda,x)\right|^4dx\right)^{3/4} \times \\ & \quad \quad \quad \quad \left(\int_{h/H}^{(h+1)/H}\left|G'(\lambda,x)\right|^4dx\right)^{1/4},
\end{align*}
and since
$$|G'(\lambda,x)|= t\left|\sum_{V/2<v\le V}\chi(\lambda+v)(x+v)^{it-1}e^{2\pi i v\alpha}\right|,$$
we have by partial summation
\begin{align}
\label{G partial}
|G'(\lambda,x)|\le tV^{-1}\max_{V/2<Q\le V}\left|\sum_{V/2<v\le Q}\chi(\lambda+v)(x+v)^{it}e^{2\pi i v\alpha} \right| \quad  \text{for} \quad x\ge 0.
\end{align}
Hence by~\eqref{H},~\eqref{G def def},~\eqref{W and G} and~\eqref{G partial}
\begin{align*}
|W|\ll &\left(\frac{t}{V}\right)^{1/4}\sum_{h\le 6HN/U} \\ &   \quad \sum_{\lambda=1}^{q}I_h(\lambda)\left(\int_{h/H}^{(h+1)/H}\max_{V/2<Q\le V}\left|\sum_{V/2<v\le Q}\chi(\lambda+v)(x+v)^{it}e^{2\pi i v\alpha} \right|^4dx \right)^{1/4}.
\end{align*}
Two applications of H\"{o}lder's inequality gives,
\begin{align*}
|W|^{4} & \le \frac{t}{V}\left(\sum_{h\le 6NH/U}\sum_{\lambda=1}^{q}I_{h}(\lambda)\right)^{2}\left(\sum_{h\le 6NH/U}\sum_{\lambda=1}^{q}I_{h}(\lambda)^2\right) \times \\ & \quad \quad \quad \left(\sum_{\lambda=1}^{q}\int_{0}^{6N/U}\max_{V/2<Q\le V}\left|\sum_{V/2<v\le Q}\chi(\lambda+v)(x+v)^{it}e^{2\pi i v\alpha} \right|^4dx  \right).
\end{align*}
Since  
$$\sum_{h\le 6NH/U}\sum_{\lambda=1}^{q}I_{h}(\lambda),$$
is equal to the number of solutions to the congruence
$$nu^{*}\equiv \lambda \pmod q,$$
with
$$ u\in  \cU, \  M-N<n\le M+N, \  1\le \lambda \le q,$$
 we have 
\begin{equation}
\label{B1}
\sum_{h\le 6NH/U}\sum_{\lambda=1}^{q}I_{h}(\lambda)\le NU.
\end{equation}
The term
$$\sum_{h\le 6NH/U}\sum_{\lambda=1}^{q}I_{h}(\lambda)^2,$$
is equal to the number of solutions to the congruence
$$n_1u_1 \equiv n_2u_2 \pmod q,$$
with 
$$u_1, u_2 \in \cU, \  \  M-N<n_1,n_2 \le M+N,$$
so that from Lemma~\ref{congruence} we have,
\begin{equation}
\label{B2}
\sum_{h\le 6NH/U}\sum_{\lambda=1}^{q}I_{h}(\lambda)^2 \le NUq^{o(1)}.
\end{equation}
From~\eqref{B1} and~\eqref{B2} we get
\begin{align*}
\label{aaa}
|W|^{4}\le \frac{t}{V}(NU)^{3}q^{o(1)}\left(\sum_{\lambda=1}^{q}\int_{0}^{6N/U}\max_{V/2<Q\le V}\left|\sum_{V/2 <v\le Q}\chi(\lambda+v)(x+v)^{it}e^{2\pi i \alpha v} \right|^4dx  \right).
\end{align*}
By Lemma~\ref{double mean value} 
\begin{align*}
&\sum_{\lambda=1}^{q}\int_{0}^{6N/U}\max_{V/2<Q\le V}\left|\sum_{V/2<v\le Q}\chi(\lambda+v)(x+v)^{it}e^{2\pi i \alpha v} \right|^4dx\ll \\ & \quad \quad \quad \quad \quad \quad  V\left(V^2q+q^{1/2}t^{-1/2} V^4\right)(qt)^{o(1)}, 
\end{align*}
so that
$$|W|^{4}\le t(NU)^{3}\left(V^2q+q^{1/2}t^{-1/2} V^4\right)(qt)^{o(1)},$$
which gives
\begin{align*}
 \left|\sum_{M<n\le M+N}\chi(n)n^{it}\right|&\le t^{1/4}N^{3/4}U^{-1/4}\left(q^{1/4}V^{-1/2}+q^{1/8}t^{-1/8} \right)(qt)^{o(1)}.
\end{align*}
Recalling the choices of $U$ and $V$ in~\eqref{U,V} gives
\begin{align*}
 \left|\sum_{M<n\le M+N}\chi(n)n^{it}\right|&\le  N^{1/2}(qt)^{3/16+o(1)}.
\end{align*}

\end{document}